\documentclass{amsart}

\newtheorem{theorem}{Theorem}
\theoremstyle{plain}

\newtheorem{corollary}{Corollary}

\newtheorem{example}{Example}

\newtheorem{lemma}{Lemma}

\newtheorem{remark}{Remark}

\numberwithin{equation}{section}

\begin{document}

\noindent

\title[Gr\"{u}ss type inequalities]{Gr\"{u}ss type inequalities in semi-inner product $C^*$-modules and applications}

\author[A.Ghazanfari \and  S.Soleimani]{A. G. Ghazanfari$^{1,*}$, S. Soleimani$^{2}$}
\address{$^{1,2}$ Department of
Mathematics, Lorestan University, P.O.Box 465, Khoramabad, Iran.}

\email{$^{1}$ghazanfari.a@lu.ac.ir, $^{2}$s.soleimani81@yahoo.com}


\keywords{Gr\"{u}ss inequality, semi-inner product $C^*$-modules.\\
\indent $^{*}$ Corresponding author.}

\begin{abstract}
Some Gr\"{u}ss type inequalities in semi-inner product modules over $C^*$-algebras for $n$-tuples of vectors are established.
Also we give their natural applications for the approximation of the discrete Fourier and the Melin transforms of bounded
linear operators on a Hilbert space.
\end{abstract}
\maketitle

\section{Introduction}
\noindent
The development of mathematical inequalities (Schwarz, triangle, Bessel, Gr\"{u}ss, Gram, Hadamard, Landau, \v{C}eby$\breve{s}$ev, Holder, Minkowsky, etc.) has experienced a surge, having been stimulated by their applications in different branches of pure and applied Mathematics.
These inequalities have been frequently used as powerful tools in obtaining bounds or estimating the errors for various approximation formulae occurring in the domains mentioned above. Therefore, any new advancement related to these fundamental facts will have a flow of important consequences in the mathematical fields where these inequalities have been used before.

For two Lebesgue integrable functions $f,g:[a,b]\rightarrow \mathbb{R}$, consider the
\v{C}eby$\breve{s}$ev functional:

\begin{equation*}
T(f,g):=\frac{1}{b-a}\int_{a}^{b}f(t)g(t)dt-\frac{1}{b-a}%
\int_{a}^{b}f(t)dt\frac{1}{b-a}\int_{a}^{b}g(t)dt.
\end{equation*}%

In 1934, G. Gr\"{u}ss \cite{gru} showed that
\begin{equation}\label{1.1}
\left\vert T(f,g)\right\vert \leq \frac{1}{4}%
(M-m)(N-n),
\end{equation}%
provided $m,M,n,N$ are real numbers with the property $-\infty <m\leq f\leq
M<\infty $ and $-\infty <n\leq g\leq N<\infty \quad \text{a.e. on }[a,b].$
The constant $\frac{1}{4}$ is best possible in the sense that it cannot be
replaced by a smaller quantity and is achieved for
\[
f(x)=g(x)=sgn \Big(x-\frac{a+b}{2}\Big).
\]
The discrete version of (\ref{1.1}) states that:
If $a\leq a_i\leq A,~b\leq b_i\leq B,~(i=1,...,n)$ where $a,A,b,B,a_i,b_i$ are real
numbers, then
\begin{equation}\label{1.2}
\left| \frac{1}{n}\sum_{i=1}^na_ib_i-\frac{1}{n}\sum_{i=1}^na_i.\frac{1}{n}\sum_{i=1}^nb_i\right|\leq \frac{1}{4}(A-a)(B-b),
\end{equation}
where the constant $\frac{1}{4}$ is the best possible for an arbitrary $n\geq 1$. Some refinements of the
discrete version of Gr\"{u}ss inequality (\ref{1.2}) are given in \cite{kech, kap}.
In \cite{dra2} some new inequalities of Schwarz and Buzano type
for $n$-tuples of vectors and applications for norm
and numerical radius inequalities for $n$-tuples of bounded linear operators are given.

In the recent years, the Gr\"{u}ss inequality (\ref{1.1}) has been investigated, applied
and generalized by many authors in different areas of mathematics, among others
in inner product spaces \cite{dra1},
in the approximation of integral transforms \cite{li} and the references therein,
 in semi-inner $*$-modules for positive linear functionals
and $C^*$-seminorms \cite{gha}, for positive maps \cite{mos1}, in
inner product modules over $H^*$-algebras and $C^*$-algebras \cite{ili}. A further extension of Gr\"{u}ss type inequality
 for Bochner integrals of vector-valued functions
in Hilbert $C^*$-modules is given in \cite{gha1}.

For an entire chapter devoted to the history of this inequality see \cite{mit}
where further references are given.

We recall some of the most important Gr\"{u}ss type discrete
inequalities for inner product spaces that are available in \cite{dra}.

\begin{theorem}\label{t1.1} Let $(H; \langle \cdot, \cdot\rangle)$ be an inner product space over $\mathbb{K};~
(\mathbb{K} = \mathbb{C},\mathbb{R}),~ x_i,~ y_i \in H,~ p_i \geq 0 ~(i = 1, . . . , n)~ (n \geq 2)$ with
$\sum_{i=1}^n p_i= 1$. If $x, X, y, Y\in H$ are such that
\begin{equation*}
 Re \left<X- x_i, x_i - x \right>\geq 0\quad and\quad Re \left<Y - y_i, y_i -y\right>\geq 0
\end{equation*}
for all $i\in \{1, . . . , n\}$, or, equivalently,
\begin{equation*}
\left\|x_i-\frac{x+X}{2}\right\|\leq\frac{1}{2}\|X-x\| \quad and \quad \left\|y_i-\frac{y+Y}{2}\right\|
\leq\frac{1}{2}\|Y-y\|
\end{equation*}
for all $i\in \{1, . . . , n\}$, then the following inequality holds
\begin{equation}\label{1.3}
\left|\sum_{i=1}^n p_i \langle x_i ,y_i \rangle-\left< \sum_{i=1}^n p_i x_i, \sum_{i=1}^n p_i y_i \right>\right|
\leq\frac{1}{4}\|X-x\|\|Y-y\|.
\end{equation}
The constant $\frac{1}{4}$ is best possible in
the sense that it cannot be replaced by a smaller quantity.
\end{theorem}

\begin{theorem}
\label{t1.2} Let $(H; \langle \cdot, \cdot\rangle)$ and $\mathbb{K}$ be as above and
$\overline{x}=(x_1, . . . , x_n)\in H^n$, $\overline{\alpha}=(\alpha_1, . . . , \alpha_n) \in \mathbb{K}^n$ and
$\overline{p}=(p_1,...,p_n)$ a probability vector. If $x,X \in H$ are
such that
\begin{equation*}
 Re \left<X- x_i, x_i - x \right>\geq 0~for~ all~ i\in \{1, . . . , n\},
\end{equation*}
or, equivalently,
\begin{equation*}
\left\|x_i-\frac{x+X}{2}\right\|\leq\frac{1}{2}\|X-x\| ~for~ all~ i\in \{1, . . . , n\},
\end{equation*}
holds, then the following inequality holds
\begin{align}\label{1.4}
\left\|\sum_{i=1}^np_i\alpha_i x_i-\sum_{i=1}^n p_i \alpha_i \sum_{i=1}^n p_ix_i\right\|
&\leq\frac{1}{2}\|X-x\|\sum_{i=1}^np_i\left|\alpha_i-\sum_{j=1}^np_j\alpha_j\right|\notag\\
&\leq\frac{1}{2}\|X-x\|\left[\sum_{i=1}^np_i|\alpha_i|^2-\left|\sum_{i=1}^np_i\alpha_i\right|^2\right]^{\frac{1}{2}}.
\end{align}
The constant $\frac{1}{2}$ in the first and second inequalities is best possible.
\end{theorem}

Motivated by the above results in the present paper, we
obtain some further generalization of Gr\"{u}ss type inequalities in semi-inner product modules over
$C^*$-algebras. we give some analogue of the discrete Gr\"{u}ss inequality (\ref{1.2})
 for $n$-tuples of vectors, which are generalizations of Theorem \ref{t1.1} and Theorem \ref{t1.2}.
We also give some their applications for the approximation of the discrete Fourier and Melin transforms.
In order to do that we need the following preliminary definitions and results.
\section{Preliminaries}
The theory of Hilbert spaces plays a central role in contemporary mathematics with numerous applications for linear operators,
matrix analysis, partial differential equations, nonlinear analysis, approximation theory, optimization theory, numerical
analysis, probability theory, statistics and other fields.
Hilbert spaces have a rich geometric structure because they are endowed with an inner product that allows the
introduction of the concept of orthogonality of vectors.

Hilbert $C^*$-modules are used as the framework for Kasparov's bivariant K-theory and form the
technical underpinning for the $C^*$-algebraic approach to quantum groups. Hilbert $C^*$-modules
are very useful in the following research areas: operator K-theory, index theory for operator-valued
conditional expectations, group representation theory, the theory of $AW^*$-algebras, noncommutative
geometry, and others. Hilbert $C^*$-modules form a category in between Banach
spaces and Hilbert spaces and obey the same axioms as a Hilbert space except that the inner
product takes values in a general $C^*$-algebra rather than the complex number $\mathbb{C}$. This simple generalization
gives a lot of trouble. Fundamental and familiar Hilbert space properties like Pythagoras' equality, self-duality and
decomposition into orthogonal complements must be given up. Moreover,
a bounded module map between Hilbert $C^*$-modules does not need to have an adjoint; not every
adjointable operator needs to have a polar decomposition. Hence to get its applications, we have to
use it with great care.

Let $\mathcal{A}$ be a $C^*$-algebra. A semi-inner product module
over $\mathcal{A}$ is a right module $X$ over $\mathcal{A}$ together with a generalized semi-inner product,
that is with a mapping $\langle.,.\rangle$ on $X\times X$, which is $\mathcal{A}$-valued and has the following properties:
\begin{enumerate}
\item[(i)]
 $\langle x,y+z\rangle=\langle x,y\rangle+\langle x,z\rangle$ for all $x,y,z\in X,$
\item[(ii)]
$\left\langle x,ya\right\rangle=\left\langle x,y\right\rangle a$ for $x,y\in X, a\in \mathcal{A}$,
\item[(iii)]
$\langle x,y\rangle ^*=\langle y,x\rangle$ for all $x,y\in X$,
\item[(iv)]
$\left\langle x,x\right\rangle \geq 0$ for $x\in X$.
\end{enumerate}
We will say that $X$ is a semi-inner product $C^*$-module.
The absolute value of $x\in X$ is defined as the square root of $\langle x,x\rangle$, and it is denoted
by $|x|$.

If, in addition,
\begin{enumerate}
\item[(v)]
$\langle x,x\rangle=0$ implies $x=0$,
\end{enumerate}
then $\langle .,.\rangle$ is called a generalized inner
product and $X$ is called an inner product module over $\mathcal{A}$ or an inner product $C^*$-module.

As we can see, an inner product module obeys the same axioms as an ordinary
inner product space, except that the inner product takes values in a more general
structure rather than in the field of complex numbers.

If $\mathcal{A}$ is a $C^*$-algebra and $X$ is a semi-inner product $\mathcal{A}$-module, then the following Schwarz
inequality holds:
\begin{equation}
\langle x,y\rangle\langle y,x\rangle\leq \|\langle x,x\rangle\|\langle y,y\rangle~(x,y\in X)\label{2.1}
\end {equation}
($e.g.$ \cite[Proposition 1.1]{lan}).

It follows from the Schwarz inequality (\ref{2.1}) that
\begin{equation}
|\langle x,y\rangle|^2\leq \|\langle x,x\rangle\|\langle y,y\rangle\quad(x,y\in X)\label{2.2}
\end{equation}
(where, for $a\in \mathcal{A}, |a|=(a^*a)^\frac{1}{2}$).
Now let $\mathcal{A}$ be a $\ast $-algebra, $\varphi $ a positive linear functional on $%
\mathcal{A}$, and let $X$ be a semi-inner $\mathcal{A}$-module. We can define a
sesquilinear form on $X\times X$ by $\sigma (x,y)=\varphi \left(
\left\langle x,y\right\rangle \right) $; the Schwarz inequality for $\sigma $
implies that
\begin{equation}
\vert \varphi\langle x,y\rangle\vert^{2}\leq \varphi\langle x,x\rangle \varphi\langle y,y\rangle.  \label{2.3}
\end{equation}%
In \cite[Proposition 1, Remark 1]{gha} the authors present two other forms of the Schwarz inequality in
semi-inner $\mathcal{A}$-module $X$, one for
a positive linear functional $\varphi$ on $\mathcal{A}$:
\begin{equation}
\varphi(\langle x,y\rangle\langle y,x\rangle)
\leq \varphi \langle x,x\rangle r\langle y,y\rangle,  \label{2.4}
\end{equation}%
where $r$ is the spectral radius, and another one for a $C^*$-seminorm $\gamma$ on $\mathcal{A}$:
\begin{equation}
(\gamma\langle x,y\rangle) ^{2}\leq \gamma\langle x,x\rangle \gamma \langle y,y\rangle.  \label{2.5}
\end{equation}%

\section{Gr\"{u}ss type inequalities in inner product $C^*$-modules}


Let $X$ be an inner product $C^*$-module and $x,y,e\in X$, and let
$\left\langle e,e\right\rangle $ be an idempotent, we put
\[
G_e(x,y):=\langle x,y\rangle-\langle x,e\rangle\langle e,y\rangle.
\]
 By \cite[Lemma 2]{gha} or, a straightforward calculation shows that
\begin{align*}
\left\langle e\left\langle e,e\right\rangle -e,e\left\langle
e,e\right\rangle -e\right\rangle & =\left\langle e\left\langle
e,e\right\rangle ,e\left\langle e,e\right\rangle \right\rangle -\left\langle
e\left\langle e,e\right\rangle ,e\right\rangle -\left\langle e,e\left\langle
e,e\right\rangle \right\rangle +\left\langle e,e\right\rangle  \\
& =\left\langle e,e\right\rangle \left\langle e,e\right\rangle \left\langle
e,e\right\rangle -\left\langle e,e\right\rangle \left\langle
e,e\right\rangle -\left\langle e,e\right\rangle \left\langle
e,e\right\rangle +\left\langle e,e\right\rangle  \\
& =0,
\end{align*}%
therefore $e\left\langle e,e\right\rangle -e=0$.
This implies that
\begin{equation*}
\left\langle e,e\right\rangle \left\langle e,x\right\rangle =\left\langle
e,x\right\rangle ,\quad \left\langle x,e\right\rangle =\left\langle
x,e\right\rangle \left\langle e,e\right\rangle .
\end{equation*}

\begin{lemma}\label{l2.2}
Let $X$ be an inner product $\mathcal{A}$-module over $C^*$-algebra $\mathcal{A}$ and $x,y,e\in X$. If
$\left\langle e,e\right\rangle $ be an idempotent in $\mathcal{A}$, then for every $a, b\in \mathcal{A}$, we have
\begin{align}
&(i)~G_e(x,x)=\langle x,x\rangle-\langle x,e\rangle\langle e,x\rangle\geq0,\label{3.4}\\
&(ii)~G_e(x,x)\leq \langle x-ea, x-ea\rangle,\label{3.5}\\
&(iii)~ G_e(x-ea,y-eb)=\langle x-ea,y-eb\rangle-\langle x-ea,e\rangle\langle e,y-eb\rangle\label{3.6}\\
&\qquad\qquad\qquad\qquad\qquad=\langle x,y\rangle-\langle x,e\rangle\langle e,y\rangle=G_e(x,y).\notag
\end{align}
\end{lemma}

\begin{proof}
By a simple calculation, we get
\[
G_e(x,x)=\langle x,x\rangle-\langle x,e\rangle\langle e,x\rangle=\langle x-e\langle e,x\rangle,x-e\langle e,x\rangle\rangle\geq0.
\]
and
\[
\langle x-ea,x-ea\rangle=\langle x-e\langle e,x\rangle, x-e\langle e,x\rangle\rangle+\langle a-e\langle e,x\rangle, a-e
\langle e,x\rangle\rangle.
\]
Therefore,
\[
G_e(x,x)=\langle x-e\langle e,x\rangle, x-e\langle e,x\rangle\rangle\leq \langle x-ea,x-ea\rangle.
\]
A straightforward calculation shows that
\begin{multline*}
G_e(x-ea,y-eb)=\langle x-ea,y-eb\rangle-\langle x-ea,e\rangle\langle e,y-eb\rangle\\
=\langle x,y\rangle-\langle x,e\rangle b-a^*\langle e,y\rangle+a^*\langle e,e\rangle b-
\left[\langle x,e\rangle-a^*\langle e,e\rangle\right]\left[\langle e,y\rangle-\langle e,e\rangle b\right]\\
=\langle x,y\rangle-\langle x,e\rangle b-a^*\langle e,y\rangle+a^*\langle e,e\rangle b\\
-\langle x,e\rangle\langle e,y\rangle+\langle x,e\rangle b+a^*\langle e,y\rangle-a^*\langle e,e\rangle b\\
=\langle x,y\rangle-\langle x,e\rangle\langle e,y\rangle=G_e(x,y).
\end{multline*}
\end{proof}

\begin{theorem}\label{t3.1}
Let $X$ be an inner product $\mathcal{A}$-module over $C^*$-algebra $\mathcal{A}$ and $x,y,e\in X$. If
$\left\langle e,e\right\rangle $ be an idempotent in $\mathcal{A}$, then for every $a, b, c, d\in \mathcal{A}$, we have
\begin{multline}\label{3.7}
|\langle x,y\rangle-\langle x,e\rangle\langle e,y\rangle|
\leq\left\| x-e\left(\frac{a+b}{2}\right)\right\|~
\left| y-e\left(\frac{c+d}{2}\right)\right|\\
=\left\|\frac{1}{4}|e(a-b)|^2-Re\langle x-ea, eb-x\rangle\right\|^\frac{1}{2}\left(\frac{1}{4}|e(c-d)|^2-Re\langle y-ec, ed-y\rangle\right)^\frac{1}{2}.
\end{multline}
furthermore, if
\[
Re\langle x-ea, eb-x\rangle\geq 0,~~ Re\langle y-ec, ed-y\rangle\geq0,
\]
then
\[
|\langle x,y\rangle-\langle x,e\rangle\langle e,y\rangle|\leq\frac{1}{4}\|e(a-b)\|~|e(c-d)|.
\]
\end{theorem}

\begin{proof}
It is easy to show that $G_e(\cdot,\cdot)$ is an $\mathcal{A}$-value semi-inner product on $X$.
Using Schwarz inequality (\ref{2.2}), we obtain
\begin{equation*}
|\langle x,y\rangle-\langle x,e\rangle\langle e,y\rangle|^2\leq\|\langle x,x\rangle-\langle x,e\rangle\langle e,x\rangle\|
\big(\langle y,y\rangle-\langle y,e\rangle\langle e,y\rangle\big).
\end{equation*}
From (\ref{3.5}), we get
\begin{equation*}
\langle x,x\rangle-\langle x,e\rangle\langle e,x\rangle
\leq\left\langle x-e\left(\frac{a+b}{2}\right),x-e\left(\frac{a+b}{2}\right)\right\rangle
\end{equation*}
and
\begin{equation*}
\langle y,y\rangle-\langle y,e\rangle\langle e,y\rangle
\leq\left\langle y-e\left(\frac{c+d}{2}\right),y-e\left(\frac{c+d}{2}\right)\right\rangle.
\end{equation*}

Since for any $y, x, x'\in X$
\begin{align*}
\dfrac{1}{4}|x'-x|^{2}-\left| y-\dfrac{x'+x}{2}\right|^{2}=Re\langle y-x', x-y\rangle,
\end{align*}
therefore, we have
\begin{equation*}
\left| x-e\left(\frac{a+b}{2}\right)\right|^2=\frac{1}{4}|e(a-b)|^2-Re\langle x-ea, eb-x\rangle,
\end{equation*}
\begin{equation*}
\left| y-e\left(\frac{c+d}{2}\right)\right|^2=\frac{1}{4}|e(c-d)|^2-Re\langle y-ec, ed-y\rangle.
\end{equation*}
The rest follows from these facts and we omit the details.
\end{proof}

\begin{example}
Let $L^1(H)$ be the set of all trace class operators on the Hilbert space $H$. It is known that $L^1(H)$ is a
Hilbert $B(H)$-module with the inner product defined by $\left\langle X,Y\right\rangle :=X^{\ast }Y$. If $E$ is a trace class operator
 such that $|E|$ is an idempotent in $B(H)$ then for every $ A, B, C, D\in B(H)$ and $X, Y\in L^1(H)$
we have
\begin{align*}
|X^*Y-X^*EY|^2&\leq\left\|\frac{1}{4}\big|E(A-B)\big|^2-Re\langle X-EA, EB-X\rangle\right\|\\
&\times\left|\frac{1}{4}\big|E(C-D)\big|^2-Re\langle Y-EC, ED-Y\rangle\right|\\
&=\left\| X-E\left(\frac{A+B}{2}\right)\right\|^2\left| Y-E\left(\frac{C+D}{2}\right)\right|^2.
\end{align*}
\end{example}

\section{Gr\"{u}ss type inequalities in semi-inner product $C^*$-modules}

Before stating the main results in this section, let us fix the rest of our notation. We assume unless stated otherwise,
throughout this paper $\mathcal{A}
$ is a $C^\ast $-algebra and $\overline{p}=(p_1,...,p_n)\in \mathbb{R}^n$ a
probability vector i.e.
$p_i \geq 0 \quad (i = 1, . . . , n)$ and $\sum_{i=1}^n p_i=1 $. If $X$ is a semi-inner product $C^*$-module and
$\overline{x}=(x_1,...,x_n), \overline{y}=(y_1,...,y_n)\in X^n$ we put
$$G_{\overline{p}}(\overline{x},\overline{y}):=\sum_{i=1}^n p_i \langle x_i ,y_i \rangle-
\left< \sum_{i=1}^n p_i x_i, \sum_{i=1}^n p_i y_i \right>.$$

\begin{lemma}\label{l2.1}
Let $X$ be a semi-inner product $C^*$-module, $a,b\in X$, $\overline{x}=(x_1,...,x_n), \overline{y}=(y_1,...,y_n)\in X^n,
\overline{\alpha}=(\alpha_1, . . . , \alpha_n)
 \in \mathbb{K}^n$; $(\mathbb{K} = \mathbb{C},\mathbb{R})$ and $\overline{p}=(p_1,...,p_n)\in \mathbb{R}^n$ a probability vector, then
\begin{equation}
\sum_{i=1}^np_i\alpha_i x_i-\sum_{i=1}^n p_i \alpha_i \sum_{i=1}^n p_ix_i=\sum_{i=1}^np_i\Big(\alpha_i-\sum_{j=1}^np_j\alpha_j\Big)
(x_i-a),\label{3.1}
\end{equation}
and
\begin{equation}
G_{\overline{p}}(\overline{x},\overline{y})=
\sum_{i=1}^n p_i\left<x_i -a, y_i -b \right>-\left< \sum_{i=1}^n p_i (x_i -a), \sum_{i=1}^n p_i (y_i-b)\right>.\label{3.2}
\end{equation}
In particular
\begin{equation}
G_{\overline{p}}(\overline{x},\overline{x})=\sum_{i=1}^n p_i\left|x_i -a\right|^2-\left| \sum_{i=1}^n p_i x_i -a\right|^2\label{3.3}
\leq\sum_{i=1}^n p_i\left|x_i -a\right|^2.
\end{equation}
\end{lemma}

\begin{proof}
For every $a\in X$ a simple calculation shows that
\begin{multline*}
\sum_{i=1}^np_i\Big(\alpha_i-\sum_{j=1}^np_j\alpha_j\Big)(x_i-a)=\sum_{i=1}^np_i\alpha_i x_i-\sum_{j=1}^n p_j
\alpha_j\sum_{i=1}^np_ix_i\notag\\
-a\sum_{i=1}^np_i\alpha_i+a\sum_{i=1}^np_i\sum_{j=1}^n p_j \alpha_j\\
=\sum_{i=1}^np_i\alpha_i x_i-\sum_{i=1}^n p_i \alpha_i \sum_{i=1}^n p_ix_i.
\end{multline*}

For every $a,b\in X$, a simple calculation shows that

\begin{multline*}
\sum_{i=1}^n p_i\left<x_i -a, y_i -b \right>-\left< \sum_{i=1}^n p_i (x_i -a), \sum_{i=1}^n p_i (y_i-b)\right>\notag\\
=\sum_{i=1}^n p_i\big(\langle x_i,y_i\rangle-\langle x_i,b\rangle-\langle a,y_i\rangle+\langle a,b\rangle\big)\\
-\left< \sum_{i=1}^n p_ix_i -a, \sum_{i=1}^n p_iy_i-b\right>\\
=\sum_{i=1}^n p_i \langle x_i ,y_i \rangle-\left< \sum_{i=1}^n p_i x_i, \sum_{i=1}^n p_i y_i \right>=
G_{\overline{p}}(\overline{x},\overline{y}).
\end{multline*}
In particular for $a=b, x_i=y_i$ we have
\begin{multline*}
G_{\overline{p}}(\overline{x},\overline{x})=\sum_{i=1}^n p_i\left<x_i -a, x_i -a \right>-\left< \sum_{i=1}^n p_i (x_i -a), \sum_{i=1}^n p_i (x_i-a)\right>\notag\\
=\sum_{i=1}^n p_i\left|x_i -a\right|^2-\left| \sum_{i=1}^n p_i (x_i -a)\right|^2\leq\sum_{i=1}^n p_i\left|x_i -a\right|^2.\\
\end{multline*}

\end{proof}
In the following Theorem we give a generalization of Theorem \ref{t1.1} for semi-inner product
$C^*$-modules.

\begin{theorem}\label{t3.2}
Let $X$ be a semi-inner product $C^*$-module, $a,b\in X$ and $\overline{p}=(p_1,...,p_n)\in \mathbb{R}^n$ a probability vector. If $\overline{x}=(x_1,...,x_n), \overline{y}=(y_1,...,y_n)\in X^n$,
then the following inequality holds
\begin{multline}\label{3.8}
\left|\sum_{i=1}^n p_i\left<x_i, y_i \right>-\left< \sum_{i=1}^n p_i x_i, \sum_{i=1}^n p_i y_i\right>\right|^2\\
\leq \left\|\sum_{i=1}^n p_i\left|x_i -a\right|^2-\left| \sum_{i=1}^n p_i x_i -a\right|^2\right\|
\left(\sum_{i=1}^n p_i\left|y_i -b\right|^2-\left| \sum_{i=1}^n p_i y_i -b\right|^2\right)\\
\leq\left(\sum_{i=1}^n p_i\left\|x_i -a\right\|^2\right)\left(\sum_{i=1}^n p_i\left|y_i -b\right|^2\right)
\end{multline}
\end{theorem}

\begin{proof}
A simple calculation shows that
\begin{equation*}
\sum_{i=1}^n p_i \langle x_i ,y_i \rangle-\left< \sum_{i=1}^n p_i x_i, \sum_{i=1}^n p_i y_i \right>=\frac{1}{2}
\sum_{i,j=1}^n p_ip_j\left<x_i -x_j ,y_i -y_j\right>,
\end{equation*}
therefore
\begin{equation*}
G_{\overline{p}}(\overline{x},\overline{x})=\frac{1}{2}\sum_{i,j=1}^n p_ip_j\left<x_i -x_j ,x_i -x_j\right>\geq 0.
\end{equation*}
It is easy to show that $G_{\overline{p}}(.,.)$ is an $\mathcal{A}$-value semi-inner product on $X^n$,
so Schwarz inequality (\ref{2.2}) holds i.e.,
\begin{equation}\label{3.9}
|G_{\overline{p}}(\overline{x},\overline{y})|^2\leq \|G_{\overline{p}}(\overline{x},\overline{x})\|
G_{\overline{p}}(\overline{y},\overline{y}).
\end{equation}
From (\ref{3.3}), we get
\begin{equation}\label{3.10}
\|G_{\overline{p}}(\overline{x},\overline{x})\|=\left\|\sum_{i=1}^n p_i\left|x_i -a\right|^2-\left|
\sum_{i=1}^n p_i x_i -a\right|^2\right\|
\leq \sum_{i=1}^n p_i\left\|x_i -a\right\|^2
\end{equation}
and
\begin{equation}\label{3.11}
G_{\overline{p}}(\overline{y},\overline{y})=\sum_{i=1}^n p_i\left|y_i -b\right|^2-\left| \sum_{i=1}^n p_i y_i -b\right|^2
\leq \sum_{i=1}^n p_i|y_i -b|^2.
\end{equation}
From  inequalities (\ref{3.9}), (\ref{3.10}) and (\ref{3.11}) we obtain the inequality (\ref{3.8}).
\end{proof}

Since every inner-product space $H$ can be regarded as an inner product $\mathbb{C}$-module, therefore the following inequality
(\ref{3.13}) is a generalization of inequality (\ref{1.3}).

\begin{corollary}\label{c3.1}
Let $X$ be a semi-inner product $C^*$-module, $a,b\in X$ and $\overline{p}=(p_1,...,p_n)\in \mathbb{R}^n$ a probability vector. If $\overline{x}=(x_1,...,x_n), \overline{y}=(y_1,...,y_n)\in X^n$, $r\geq 0, s\geq 0$ are such that
\begin{equation}
\|x_i-a\|\leq r,\quad \|y_i-b\|\leq s,\text{ for all i }\in \{1, . . . , n\},\label{3.12}
\end{equation}
then the following inequality holds
\begin{equation}
\left\|\sum_{i=1}^n p_i\left<x_i, y_i \right>-\left< \sum_{i=1}^n p_i x_i, \sum_{i=1}^n p_i y_i\right>\right\|\leq rs.\label{3.13}
\end{equation}
The constant 1 coefficient of $rs$ in the inequality (\ref{3.13}) is best possible in
the sense that it cannot be replaced by a smaller quantity.
\end{corollary}

\begin{proof}
From inequalities (\ref{3.8}) and (\ref{3.12}) we obtain (\ref{3.13}).

To prove the sharpness of the constant 1 in the inequality in (\ref{3.13}), let us assume that, under the assumptions of
the theorem, the
inequalities hold with a constant $c>0$, i.e.,

\begin{equation}
\|G_{\overline{p}}(\overline{x},\overline{y})\|\leq crs \label{3.14}.
 \end{equation}
Assume that $n = 2, p_1 = p_2 = \frac{1}{2}$ and $e$ is an element of $X$ such that $\|\langle e,e\rangle\|=1$.
We put
\begin{align*}
x_1&=a+re,~~~y_1=b+se\\
x_2&=a-re,~~~y_2=b-se,
\end{align*}
then, obviously,
\begin{equation*}
\|x_i-a\|\leq r,\quad \|y_i-b\|\leq s,\quad (i=1,2),
\end{equation*}
which shows that the condition (\ref{3.12}) holds. If we replace $n, p_1, p_2, x_1, x_2, y_1, y_2$ in (\ref{3.14}), we obtain
\begin{equation*}
\|G_{\overline{p}}(\overline{x},\overline{y})\|=rs\leq crs,
\end{equation*}
from where we deduce that $c\geq 1$, which proves the sharpness of the constant 1.
\end{proof}

The following Remark \ref{r3.1}(ii) is a generalization of Theorem \ref{t1.2} for semi-inner product $C^*$-modules.

\begin{remark}\label{r3.1}\rm{
\begin{enumerate}
\item[(i)]
Let $\mathcal{A}$ be a $C^{\ast}$-algebra, and $\overline{p}=(p_1,...,p_n)\in \mathbb{R}^n$ a probability vector. If $a, b, a_i, b_i, (i=1,2,...,n) \in \mathcal{A}, r\geq 0, s\geq0$ are such that
\begin{equation*}
\|a_i-a\|\leq r,~~\|b_i-b\|\leq s, \text{ for all i }\in \{1, . . . , n\},
\end{equation*}
it is known that $\mathcal{A}$ is a Hilbert $C^*$-module over
itself with the inner product defined by $\left\langle a,b\right\rangle :=a^{\ast }b$.
In this case (\ref{3.13}) implies that
\begin{align*}
\left\|\sum_{i=1}^n p_ia_i^*b_i-\sum_{i=1}^n p_ia_i^*.\sum_{i=1}^n p_ib_i\right\|\leq rs.
\end{align*}
Since
\begin{equation*}
\|a_i^*-a^*\|\leq r,~~ \text{ for all i }\in \{1, . . . , n\},
\end{equation*}
we deduce
\begin{align*}
\left\|\sum_{i=1}^n p_ia_ib_i-\sum_{i=1}^n p_ia_i.\sum_{i=1}^n p_ib_i\right\|\leq rs.
\end{align*}

\item[(ii)] Let $X$ be a semi-inner product $C^{\ast}$-module, $a\in X, \overline{\alpha}=(\alpha_1, . . . , \alpha_n)
 \in \mathbb{K}^n$ and $\overline{p}=(p_1,...,p_n)\in \mathbb{R}^n$ a probability vector. If $\overline{x}=(x_1,...,x_n)\in X^n, r\geq 0$
 are such that
\begin{equation*}
\|x_i-a\|\leq r,\text{ for all i }\in \{1, . . . , n\},
\end{equation*}
holds, from equality (\ref{3.1}) we obtain
\begin{align}
\left\|\sum_{i=1}^np_i\alpha_i x_i-\sum_{i=1}^n p_i \alpha_i \sum_{i=1}^n p_ix_i\right\|
&\leq r\sum_{i=1}^np_i\left|\alpha_i-\sum_{i=1}^np_j\alpha_j\right|\label{3.15}\\
&\leq r\left[\sum_{i=1}^np_i|\alpha_i|^2-\left|\sum_{i=1}^np_i\alpha_i\right|^2\right]^{\frac{1}{2}}.\notag
\end{align}
The constant 1 in the first and second inequalities in (\ref{3.15}) is best possible.
Since every Hilbert space is a Hilbert $\mathbb{C}$-module, the inequality (\ref{3.15}) is a generalization of (\ref{1.4}).
\end{enumerate}}
\end{remark}

\section{Applications}

In this section we give applications of Theorem \ref{t3.2} for the approximation of some discrete transforms such as the
discrete Fourier and the Melin transforms for bounded linear operators on a Hilbert space.

Let $X$ be a semi-inner product $C^*$-module on $C^*$-algebra $\mathcal{A}$ and  $x=(x_1,...,x_n), y=(y_1,...,y_n)\in X^n$.
For a given $\omega\in \mathbb{R}$,
define the \textit{discrete Fourier transform}

\begin{equation*}
\mathcal{F}_\omega(x)(m)=\sum_{k=1}^n \exp(2\omega imk)\times x_k,\quad m=1,...,n.
\end{equation*}
The element $\sum_{k=1}^n \exp(2\omega imk)\times \langle x_k, y_k\rangle$ of $\mathcal{A}$ is called
Fourier transform of the vector $(\langle x_1,y_1\rangle,...,\langle x_k,y_k\rangle)\in \mathcal{A}^n$
and will be denoted by
\begin{equation*}
\mathcal{F}_\omega(x,y)(m)=\sum_{k=1}^n \exp(2\omega imk)\times \langle x_k,y_k\rangle\quad m=1,...,n.
\end{equation*}
We can also consider the \textit{Mellin transform}
\begin{equation*}
\mathcal{M}(x)(m)=\sum_{k=1}^n k^{m-1}x_k, \quad m=1,...,n.
\end{equation*}
of the vector $x=(x_1,...,x_n) \in X^n$.\\
The Mellin transform of the vector $(\langle x_1,y_1\rangle,...,\langle x_k,y_k\rangle)\in \mathcal{A}^n$ is defined by\\
$\sum_{k=1}^n k^{m-1}\langle x_k,y_k\rangle$ and will
be denoted by
\begin{equation*}
\mathcal{M}(x,y)(m)=\sum_{k=1}^n k^{m-1}\langle x_k,y_k\rangle.
\end{equation*}

\begin{example}
Let $A_1,...,A_n$ and $B_1,...,B_n$ be bounded linear operators in $B(H_1, H_2)$. It is known that $B(H_1, H_2)$
is a Hilbert $C^*$-module over $B(H_1)$ with the inner product defined by $\left\langle X,Y\right\rangle :=X^{\ast }Y$, therefore for every $A, B\in B(H_1)$, from the inequality (\ref{3.8})
we obtain

\begin{multline}\label{4.1}
\left|\sum_{k=1}^n\exp (2\omega imk)A_{k}^*B_k-\left(\frac{1}{n}\sum_{k=1}^n A_{k}^*\right)\left(\sum_{k=1}^n\exp (2\omega imk)B_k\right)\right|^2\\
\leq\left\|\sum_{k=1}^n|A_k-A|^2-\left|\frac{1}{n}\sum_{k=1}^nA_k-A\right|^2\right\|\\
\times\left(\sum_{k=1}^n|\exp (2\omega imk)B_k-B|^2-\left|\frac{1}{n}\sum_{k=1}^n\exp (2\omega imk)B_k-B\right|^2\right)\\
\leq\left(\sum_{k=1}^n\|A_k-A\|^2\right)\left(\sum_{k=1}^n|\exp (2\omega imk)B_k-B|^2\right).
\end{multline}

and
\begin{multline}\label{4.2}
\left|\sum_{k=1}^nk^{m-1}A_{k}^*B_k-\left(\frac{1}{n}\sum_{k=1}^n A_{k}^*\right)\left(\sum_{k=1}^nk^{m-1}B_k\right)\right|^2\\
\leq\left\|\sum_{k=1}^n|A_k-A|^2-\left|\frac{1}{n}\sum_{k=1}^nA_k-A\right|^2\right\|\\
\times\left(\sum_{k=1}^n|k^{m-1}B_k-B|^2-\left|\frac{1}{n}\sum_{k=1}^nk^{m-1}B_k-B\right|^2\right)\\
\leq\left(\sum_{k=1}^n\|A_k-A\|^2\right)\left(\sum_{k=1}^n|k^{m-1}B_k-B|^2\right).
\end{multline}
\end{example}

\begin{example}
Let $B,A,A_1,...,A_n$ be bounded linear operators on the Hilbert space $H$, $\overline{\alpha}=(\alpha_1,...,\alpha_n)\in \mathbb{K}^n$.
and $I$ be identity operator on $H$.
\begin{multline}\label{4.3}
\left|\sum_{k=1}^n p_k\alpha_k A_k -\left( \sum_{k=1}^n p_k \alpha_k\right)\left( \sum_{k=1}^n p_k A_k\right)\right|^2\\
=\left|\sum_{k=1}^n p_k\left<\overline{\alpha_k}I, A_k \right>-\left< \sum_{k=1}^n p_k \overline{\alpha_k}I,
\sum_{k=1}^n p_k A_k\right>\right|^2\\
\leq \left\|\sum_{k=1}^n p_k\left|\overline{\alpha}_kI -A\right|^2-\left| \sum_{k=1}^n p_k \overline{\alpha}_kI -A\right|^2\right\|
\left(\sum_{k=1}^n p_k\left|A_k -B\right|^2-\left| \sum_{k=1}^n p_k A_k -B\right|^2\right),
\end{multline}
for $A=B=0$ we get
\begin{multline}\label{4.4}
\left|\sum_{k=1}^n p_k\alpha_k A_k -\left( \sum_{k=1}^n p_k \alpha_k\right)\left( \sum_{k=1}^n p_k A_k\right)\right|^2\\
\leq \left(\sum_{k=1}^n p_k\left|\alpha_k\right|^2-\left| \sum_{k=1}^n p_k \alpha_k\right|^2\right)
\left(\sum_{k=1}^n p_k\left|A_k\right|^2-\left| \sum_{k=1}^n p_k A_k\right|^2\right).
\end{multline}
\end{example}

A simple calculation shows that (see the proof of Theorem 59 in \cite{dra}),
\begin{equation*}
\sum_{k=1}^n \exp(2\omega imk)=\frac{\sin(\omega mn)}{\sin(\omega m)}\times \exp[\omega (n+1)im].
\end{equation*}
Putting $\alpha_k=\exp(2\omega imk),~p_k=\frac{1}{n}$, in (\ref{4.4}), we get
\begin{multline}\label{4.5}
\left|\sum_{k=1}^n\exp (2\omega imk)A_{k}-\frac{\sin(\omega mn)}{\sin(\omega m)}\exp[\omega(n+1)im]\times \frac{1}{n}
\sum_{k=1}^n A_k\right|^2\\
\leq \left[n^2-\frac{\sin^2(\omega mn)}{\sin^2(\omega m)}\right]\left[\frac{1}{n}\sum_{k=1}^n|A_k|^2-\frac{1}{n^2}\left|\sum_{k=1}^nA_k\right|^2\right].
\end{multline}
Also, Putting $\alpha_k=k^{m-1},~p_k=\frac{1}{n}$, in (\ref{4.4}), we obtain
\begin{multline}\label{4.6}
\Big|\sum_{i=1}^nk^{m-1}A_k-S_{m-1}(n).\frac{1}{n}\sum_{k=1}^n A_k\Big|^2\\
\leq \big[nS_{2m-2}(n)-S_{m-1}^2(n)\big]\left[\frac{1}{n}\sum_{k=1}^n|A_k|^2-\frac{1}{n^2}\left|\sum_{k=1}^nA_k\right|^2\right]
, m\in\{1,...,n\},
\end{multline}
where $S_p(n), p\in \mathbb{R}, n\in \mathbb{N}$ is the $p$-powered sum of the first $n$ natural
numbers, i.e.,
\begin{equation*}
S_p(n):=\sum_{k=1}^n k^p.
\end{equation*}
For the following particular values of Mellin Transform (see \cite[Corollary 4]{dra1}), we have
\begin{multline*}
\left|\sum_{k=1}^n kA_k-\frac{n+1}{2}\sum_{k=1}^n A_k\right|^2\leq \left[\frac{n^2(n-1)(n+1)}{12}\right]\left[\frac{1}{n}\sum_{k=1}^n|A_k|^2-\frac{1}{n^2}\left|\sum_{k=1}^nA_k\right|^2\right] \\
\leq\left[\frac{n(n-1)(n+1)}{12}\right]\sum_{k=1}^n|A_k|^2,
\end{multline*}
and
\begin{multline*}
\Big|\sum_{k=1}^n k^2A_k-\frac{(n+1)(2n+1)}{6}\sum_{k=1}^n A_k\Big|^2\\
\leq\left(\frac{n^2(n-1)(n+1)(2n+1)(8n+11)}{180}\right)\left[\frac{1}{n}\sum_{k=1}^n|A_k|^2-\frac{1}{n^2}\left|
\sum_{k=1}^nA_k\right|^2\right]\\
\leq\left(\frac{n(n-1)(n+1)(2n+1)(8n+11)}{180}\right)\sum_{k=1}^n|A_k|^2.
\end{multline*}

There exist other examples for the approximation of
some discrete transforms such as polynomials with coefficients in a semi-inner product $C^*$-module.
However, the details are omitted but each of them can be proven
 in a similar manner as this section.



\end{document}